\pdfoutput=1

\documentclass[12pt,a4paper,reqno]{amsart} 
\usepackage{amsmath,amsthm,amssymb,mathrsfs,mathtools} 
\usepackage{microtype,fixltx2e,lmodern} 
\usepackage{enumerate,comment,braket,xspace,tikz-cd,csquotes} 
\usepackage[utf8]{inputenc} 
\usepackage[T1]{fontenc} 
\usepackage[centering,vscale=0.7,hscale=0.7]{geometry}
\usepackage[hidelinks]{hyperref}

\newcommand*{\longhookrightarrow}{\ensuremath{\lhook\joinrel\relbar\joinrel\rightarrow}}

\newcommand{\R}{\mathbb R}
\newcommand{\Q}{\mathbb Q}
\newcommand{\fX}{\mathfrak X}
\newcommand{\IX}{{I_\fX}}
\newcommand{\fXbs}{\fX_{\overline s}}
\newcommand{\fXbe}{\fX_{\overline \eta}}
\newcommand{\LX}{\Lambda_{\fX}}
\newcommand{\LXe}{\Lambda_{\fX_\eta}}
\newcommand{\LXs}{\Lambda_{\fXbs}}
\newcommand{\QXe}{\Q_{\ell,\fX_\eta}}

\newcommand{\fU}{\mathfrak U}

\newcommand{\QUe}{\Q_{\ell,\fU_\eta}}
\newcommand{\fB}{\mathfrak B}
\newcommand{\tfB}{\widetilde{\fB}}
\newcommand{\fS}{\mathfrak S}
\newcommand{\Z}{\mathbb Z}
\newcommand{\et}{_\mathrm{\acute{e}t}}
\newcommand{\wD}{\widetilde D}
\newcommand{\an}{^\mathrm{an}}
\newcommand{\Gm}{\mathbb G_\mathrm{m}}
\newcommand{\Gmnan}{(\Gm^n)\an}
\newcommand{\trop}{^\mathrm{t}}

\DeclareMathOperator{\dist}{dist}

\DeclareMathOperator{\val}{val}
\DeclareMathOperator{\Spec}{Spec}
\DeclareMathOperator{\Spf}{Spf}

\DeclareMathOperator{\Image}{Im}
\DeclareMathOperator{\Ker}{Ker}
\DeclareMathOperator{\Coker}{Coker}

\DeclareMathOperator{\Bl}{Bl}
\DeclareMathOperator{\CH}{CH}

\theoremstyle{plain}
\newtheorem{thm}{Theorem}[section]
\newtheorem{lem}[thm]{Lemma}

\newtheorem{prop}[thm]{Proposition}
\newtheorem{cor}[thm]{Corollary}
\theoremstyle{definition}
\newtheorem{defin}[thm]{Definition}
\theoremstyle{remark}
\newtheorem{eg}[thm]{Example}
\newtheorem{rem}[thm]{Remark}

\numberwithin{equation}{section}

\title{Balancing conditions in global tropical geometry}
\author{Tony Yue YU}
\date{8 Apr 2013 (Revised on 22 Jan 2015)}
\address{Tony Yue YU, Institut de Math\'ematiques de Jussieu, CNRS-UMR 7586, Case 7012, Universit\'e Paris Diderot - Paris 7, B\^atiment Sophie Germain 75205 Paris Cedex 13 France}
\email{yuyuetony@gmail.com}

\begin{document}

\begin{abstract}
We study tropical geometry in the global setting using Berkovich's deformation retraction.
We state and prove the generalized balancing conditions in this setting.
Starting with a strictly semi-stable formal scheme, we calculate certain sheaves of vanishing cycles using analytic étale cohomology, then we interpret the tropical weights via these cycles.
We obtain the balancing condition for tropical curves on the skeleton associated to the formal scheme in terms of the intersection theory on the special fiber.
Our approach works over any complete discrete valuation field.
\end{abstract}

\maketitle

\section{Introduction and statement of result}\label{sec:intro}

Tropicalization is a procedure which relates algebraic geometry with tropical geometry.
Usually, tropicalization is carried out in the setting of toric varieties \cite{Einsiedler_Non-archimedean_2006,Baker_Nonarchimedean_2011,Gubler_Guide_2013}.
Let us give a quick review in the case of curves.

By definition, a toric variety contains an open dense torus $\Gm^n$.
Assume that our torus $\Gm^n$ is defined over a complete discrete valuation field $k$.
Consider the coordinate-wise valuation map
\begin{equation}\label{eq:tau_0}
\Gm^n(k)\simeq (k^\times)^n \rightarrow \R^n, \qquad
(x_1,\dots,x_n) \mapsto (\val(x_1),\dots,\val(x_n)).
\end{equation}
This map can be extended to the Berkovich analytification $\Gmnan$ and we obtain a continuous surjective map $\tau_0\colon \Gmnan\rightarrow\R^n$.

Now let $C$ be an analytic curve embedded in $\Gmnan$.
Traditionally, the tropicalization of $C$ is by definition the image $\tau_0(C)\subset\R^n$, which we denote by $C\trop$.
The tropical curve $C\trop$ has the structure of a metrized graph satisfying the \emph{balancing condition}.
One idea of tropical geometry is to study curves in algebraic varieties in terms of such combinatorial gadgets.

Let us recall the classical balancing condition.
To each edge $e$ of $C\trop$ with a chosen orientation, one can associate an element $\widetilde w_e\in \Z^n$, parallel to the direction of $e$ inside $\R^n$, called the \emph{tropical weight} of $e$.
Then the balancing condition states that for any vertex $v\in C\trop$, we have $\sum_{e\ni v} \widetilde w_e=0$, where the sum is taken over all edges that contain $v$ as an endpoint, and the orientation of each edge is chosen to be the one that points away from $v$.
We refer to \cite{Mikhalkin_Enumerative_2005,Nishinou_Toric_2006,Speyer_Uniformizing_2007,Baker_Nonarchimedean_2011} for the proofs.

In order to go beyond the toric case above, we propose to replace the map $\tau_0\colon\Gmnan\to\R^n$ by the deformation retraction of a non-archimedean analytic space onto its skeleton constructed by V.\ Berkovich \cite{Berkovich_Smooth_1999}.
Let $\fX$ be a strictly semi-stable formal scheme over the ring of integers $k^\circ$ (see Definition \ref{def:strictly semi-stable formal scheme}).
Its generic fiber $\fX_\eta$ is a Berkovich space over $k$ and its special fiber $\fX_s$ is a scheme over the residue field $\widetilde{k}$.
Let $S_\fX$ denote the dual intersection complex of $\fX_s$.
Following \cite{Berkovich_Smooth_1999}, one can construct an embedding $S_\fX\subset\fX_\eta$, and a continuous proper surjective retraction map $\tau\colon \fX_\eta\rightarrow S_\fX$.
So $S_\fX$ is called a \emph{skeleton} of $\fX_\eta$.
We consider the map $\tau\colon\fX_\eta\to S_\fX$ to be the globalization of the map $\tau_0\colon\Gmnan\to\R^n$ in \eqref{eq:tau_0}.

Now in parallel, let $C$ be a compact quasi-smooth\footnote{The quasi-smoothness assumption on the curve $C$ is not restrictive because we are considering morphisms from $C$ to $\fX_\eta$ rather than embedded curves, and one can always make desingularizations.} $k$-analytic curve, and let $f \colon C\rightarrow\fX_\eta$ be a $k$-analytic morphism.
We call the image $(\tau\circ f)(C)\subset S_\fX$ the associated tropical curve, and denote it again by $C\trop$.
By working locally, one can show as in the toric case that $C\trop$ is a graph piecewise linearly embedded in $S_\fX$.
However, the balancing condition for the tropical curve $C\trop$ is no longer clear in the global setting, because the vertices of $C\trop$ may sit on the corners of $S_\fX$.
So we ask the following question.

\bigskip
\paragraph{\textbf{Question}} Let $v$ be a vertex of the tropical curve $C\trop$. What are the constraints on the shape of $C\trop$ near the vertex $v$?

\bigskip

In this paper, we give a necessary condition in terms of the intersection theory on the special fiber.

Let $\{D_i\}_{i\in I_\fX}$ denote the set of irreducible components of the special fiber $\fX_s$.
We have a natural embedding $S_\fX\subset\R^{I_\fX}$.
Assume that the vertex $v$ sits in the relative interior of the face of $S_\fX$ corresponding to a subset $I_v\subset I_\fX$.
Let $D_{I_v}$ denote the corresponding closed stratum of $\fX_s$, and let $\overline{D}_{I_v}$ denote the base change to the algebraic closure of $\widetilde k$.

As in the classical case, to each edge $e$ of $C\trop$ containing $v$ as an endpoint, we can associate a weight $\widetilde w_e\in \Z^\IX$ (see Section \ref{sec:weights}, compare \cite[§6]{Baker_Nonarchimedean_2011}).
We denote the sum of weights around $v$ by
\begin{equation}\label{eq:sum of weights}
\sigma_v \coloneqq \sum_{e\ni v}\widetilde w_e\in \Z^\IX .
\end{equation}

Let $L_i$ be the pullback of the line bundle $\mathcal O(D_i)$ to $\overline D_{I_v}$, for every $i\in\IX$.
Let $\alpha$ be the map
\begin{align*}
\alpha \colon \CH_1\left(\overline D_{I_v}\right) &\longrightarrow \Z^\IX \\
L &\longmapsto \big(\, L\cdot L_i ,\ i\in \IX\big) ,
\end{align*}
that takes a one-dimensional cycle $L$ in $\overline D_{I_v}$ to its intersection numbers with the divisors $L_i$ for every $i\in\IX$.

Our generalized balancing condition is stated in the following theorem.
The rest of the paper provides a proof of the theorem.

\begin{thm}\label{thm:balancing condition}
Assume that the closed stratum $D_{I_v}$ is projective, and that the vertex $v$ does not lie in the image of the boundary $(\tau\circ f)(\partial C)$.
Then the sum of weights $\sigma_v$ lies in the image of the map
\[\alpha_\Q\coloneqq\alpha\otimes\Q\colon \CH_1\left(D_{I_v}\right)_\Q \longrightarrow \Q^\IX .\]
\end{thm}

\bigskip


\begin{eg}\label{eg@top_dim}
Assume that $\fX$ is $n$-dimensional and that the vertex $v$ sits in the interior of an $n$-dimensional face of $S_\fX$.
Then $D_{I_v}$ is a point and the map $\alpha$ is zero.
Our balancing condition in this case simply states that the sum of weights $\sigma_v$ must be zero.
So we recover the classical balancing condition in our generalized setting.
\end{eg}


\begin{eg}
Let us work out a concrete example for a degeneration of K3 surfaces.
Let $k=\mathbb C(\!(t)\!)$ be the field of formal Laurent series.
Let $\fX^0\subset\mathbf P^3_{\mathbb C[\![t]\!]}$ be the formal scheme given by the equation
\[x_0 x_1 x_2 x_3 + t P_4(x_0,x_1,x_2,x_3)=0,\]
where $P_4$ is a generic homogeneous polynomial of degree four. We think of $\fX^0$ as a formal family of complex K3 surfaces.
The special fiber $\fX^0_s$ consists of the four coordinate hyperplanes in $\mathbf P^3_{\mathbb C}$, which we denote by $D_0, D_1, D_2, D_3$ respectively.
The formal scheme $\fX^0$ is not strictly semi-stable.
We make a small resolution (cf.\ \cite{Atiyah_Analytic_1958}) at each of the 24 conical singularities $p_\alpha$, given by the equations
\[P_4(x_0,x_1,x_2,x_3)=0,\ x_i=x_j=0,\quad\text{for } 0\leq i<j\leq 3 .\]
More precisely, we blow up the divisors $D_0, D_1, D_2, D_3$ subsequently, and denote by $\fX$ the formal scheme after the blow-ups.
Other choices of small resolutions are also possible.
We make this particular choice for the simplicity of exposition.

Now the formal scheme $\fX$ is strictly semi-stable.
Its special fiber $\fX_s$ has four irreducible components, which are strict transforms of the divisors $D_0, D_1, D_2, D_3$. We denote them by $\wD_0, \wD_1, \wD_2, \wD_3$ respectively. We have
\begin{align*}
\wD_0 & \simeq\mathbf P_{\mathbb C}^2 ,\\
\wD_1 & \simeq \Bl_\text{\{4 points\}}\mathbf P_{\mathbb C}^2 ,\\
\wD_2 & \simeq \Bl_\text{\{8 points\}}\mathbf P_{\mathbb C}^2 ,\\
\wD_3 & \simeq \Bl_\text{\{12 points\}}\mathbf P_{\mathbb C}^2 ,
\end{align*}
where the symbol $\Bl$ means blow-up.
The dual intersection complex $S_\fX\subset\R^4$ is a hollow tetrahedron, which is homeomorphic to the sphere $S^2$.
Let $v$ be a point in $S_\fX$. Let us describe explicitly our balancing condition at $v$.
According to Theorem \ref{thm:balancing condition}, it suffices to calculate the map $\alpha \colon \CH_1\big(D_{I_v}\big)\longrightarrow \Z^4$.
We distinguish three cases.

First, the point $v$ sits in the relative interior of a $2$-dimensional face of $S_\fX$.
Then we are in the situation of Example \ref{eg@top_dim}.
So the map $\alpha$ is zero in this case.

Second, the point $v$ sits on a vertex of $S_\fX$.
Suppose for example $v$ corresponds to the divisor $\wD_0$.
We have $\CH_1(\wD_0)\simeq \CH_1(\mathbf P_{\mathbb C}^2)\simeq\Z$.
The map $\alpha\colon \CH_1(\wD_0)\rightarrow\Z^4$ sends $1$ to $(-3,1,1,1)$.
We omit the cases where $v$ corresponds to other divisors.

Third, the point $v$ sits in the relative interior of a $1$-dimensional face of $S_\fX$.
Suppose for example the corresponding closed stratum $D_{I_v}$ is the intersection $\wD_0\cap \wD_1$, which we denote by $\wD_{01}$.
It is isomorphic to the projective line $\mathbf P^1_{\mathbb C}$, so $\CH_1(\wD_{01})\simeq\Z$.
Let $E_1,E_2,E_3,E_4$ denote the four exceptional curves in $\wD_1$.
Let $\pi_1\colon\wD_1\to\mathbf P_{\mathbb C}^2$ denote the blow-up.
Then $(\pi_1^* \mathcal O(1), [ E_1], [ E_2], [ E_3], [ E_4])$ form a basis of $\CH_1(\wD_1)$.
From the relations
\begin{align*}
\big(\big[\wD_0\big] + \big[\wD_1\big] + \big[\wD_2\big] + \big[\wD_3\big]\big)_{\big|\big[\wD_0\big]} &=0\in \CH_1\big(\wD_0\big),\\
\big(\big[\wD_0\big] + \big[\wD_1\big] + \big[\wD_2\big] + \big[\wD_3\big]\big)_{\big|\big[\wD_1\big]}  &=0\in \CH_1\big(\wD_1\big) ,
\end{align*}
we have
\begin{align*}
\big[\wD_0\big]_{\big|\big[\wD_0\big]} &= -3 \in \CH_1\big(\wD_0\big),\\
\big[\wD_1\big]_{\big|\big[\wD_1\big]} &= (-3,1,1,1,1) \in \CH_1\big(\wD_1\big) .
\end{align*}
Therefore,
\begin{align*}
\big[\wD_0\big]_{\big|\big[\wD_{01}\big]}&=-3\in \CH_1\big(\wD_{01}\big),\\
\big[\wD_1\big]_{\big|\big[\wD_{01}\big]}&=-3+1+1+1+1 = 1\in \CH_1\big(\wD_{01}\big) .
\end{align*}
It is obvious that
\[\big[\wD_2\big]_{\big|\big[\wD_{01}\big]}=\big[\wD_3\big]_{\big|\big[\wD_{01}\big]}=1\in \CH_1\big(\wD_{01}\big) .\]
So we conclude that the map $\alpha\colon \CH_1\big(\wD_{01}\big)\rightarrow\Z^4$ sends $1$ to $(-3,1,1,1)$.
\end{eg}

\bigskip
\paragraph{\textbf{Plan}} Basic definitions are given in Section \ref{sec:skeleton}.
In Section \ref{sec:vanishing cycles}, we study the geometry of strictly semi-stable formal schemes in terms of vanishing cycles. In Section \ref{sec:weights}, we define tropical weights.
We prove that they are homological in nature.
Indeed, they are only related to the ``vanishing part'' of the first degree cohomology of the generic fiber (Proposition \ref{prop:weights-cycles}).
In Section \ref{sec:tube}, we establish an important technical step which allows us to localize our calculation of vanishing cycles to a smaller domain inside the skeleton.
In Section \ref{sec:cohomological balancing}, we prove a weaker form of our balancing conditions in terms of étale cohomology.
The key ingredient is the long exact sequence relating nearby cycles with vanishing cycles.
In Section \ref{sec:From cohomological classes to algebraic cycles}, we explain how to use standard arguments in algebraic geometry to obtain the stronger balancing condition (Theorem \ref{thm:balancing condition}) which is stated in terms of algebraic cycles.

\medskip
\paragraph{\bf Acknowledgement} I am very grateful to Maxim Kontsevich for inspiring discussions, from which this article originates. Discussions with Antoine Ducros, Pierrick Bousseau, Jean-François Dat, Ilia Itenberg, Sean Keel, Bernhard Keller, Bruno Klingler and Grigory Mikhalkin are equally very essential and useful.
I would also like to thank the referees for valuable comments.

\section{Strictly semi-stable formal schemes and skeleta}\label{sec:skeleton}

In this article, $k$ always denotes a complete discrete valuation field\footnote{We assume the valuation to be discrete because the theory of special formal schemes (cf.\ \cite{Berkovich_Vanishing_II_1996}) in the non-discrete valued case is not studied in the literature. It is not clear whether one can drop this assumption.}. Let $k^\circ$ be the ring of integers of $k$, $k^{\circ\circ}$ the maximal ideal of $k^\circ$, and $\widetilde k$ the residue field. The symbol $\ell$ always denotes a prime number invertible in the residue field $\widetilde k$.

For $n\geq 1$, $0\leq d\leq n$ and $a\in k^{\circ\circ}\setminus 0$, put
\begin{equation}\label{eq:standard formal scheme}
\fS(n,d,a) = \Spf \left(k^\circ\{T_0,\dots,T_{d}, S^\pm_{d+1},\dots,S^\pm_n\}/(T_0\cdots T_{d}-a)\right) .
\end{equation}

\begin{defin}
A formal scheme $\fX$ over $k^\circ$ is said to be \emph{finitely presented} if it is a finite union of open affine subschemes of the form \[\Spf\left(k^\circ\{T_0,\dots,T_n\}/(f_1,\dots,f_m)\right) .\]
\end{defin}

\begin{defin}\label{def:strictly semi-stable formal scheme}
Let $\fX$ be a formal scheme finitely presented over $k^\circ$. $\fX$ is said to be \emph{strictly semi-stable} if every point $x$ of $\fX$ has an open affine neighbourhood $\fU$ such that the structural morphism $\fU\rightarrow\Spf k^\circ$ factorizes through an étale morphism $\phi\colon \fU\rightarrow\mathfrak S(n,d,a)$ for some $0\leq d\leq n$ and $a\in k^{\circ\circ}\setminus 0$.
\end{defin}

Recall that for a formal scheme $\fX$ finitely presented over $k^\circ$, its special fiber $\fX_s$ is a scheme of finite type over $\widetilde k$, and its generic fiber $\fX_\eta$ is a compact strictly $k$-analytic space (cf.\ \cite{Berkovich_Spectral_1990,Berkovich_Etale_1993,Berkovich_Vanishing_1994}).
When $\fX$ is polystable in the sense of \cite{Berkovich_Smooth_1999}, one can construct a polysimplicial set $\mathbf C(\fX_s)$.
Its topological realization is denoted by $S_\fX$.
In \cite{Berkovich_Smooth_1999}, Berkovich constructed an embedding $S_\fX\subset\fX_\eta$ and a strong deformation retraction from $\fX_\eta$ to $S_\fX$.
So $S_\fX$ is called the \emph{skeleton} of $\fX_\eta$ with respect to $\fX$.
In our simplified situation, i.e. when $\fX$ is strictly semi-stable, the skeleton $S_\fX$ has a simple description as the dual intersection complex of the special fiber $\fX_s$.

Let $\Set{D_i | i\in\IX=\{0,\dots,N\}}$ be the set of irreducible components of the special fiber $\fX_s$. For any non-empty subset $I\subset\IX$, let $D_I=\bigcap_{i\in I} D_i$ and
\begin{equation}J_I=\Set{j | D_{I\cup \{j\}}\neq \emptyset} .\label{eq:J}\end{equation}
We further assume that the strata $D_I$ are all irreducible. The general constructions in \cite{Berkovich_Smooth_1999} imply the following two lemmas.

\begin{lem}
The skeleton $S_\fX$ is the finite simplicial sub-complex of the simplex $\Delta^{\IX}$ such that for any $I\subset\IX$, $\Delta^I$ is a face of $S_\fX$ if and only if $D_I\neq\emptyset$.
\end{lem}

A face $\Delta^I\subset S_\fX$ for $I\subset\IX$ is called \emph{maximal} if it does not belong to another face of higher dimension.
Let $\Delta^I$ be a maximal face of $S_\fX$ of dimension $d$. By Definition \ref{def:strictly semi-stable formal scheme}, there exists an affine open subscheme $\fU$ in $\fX$ such that the structural morphism $\fU\rightarrow\Spf k^\circ$ factorizes through an étale morphism $\phi\colon\fU\rightarrow\fS(n,d,a)$ for some $n\in\mathbb N$, $a\in k^{\circ\circ}\setminus 0$, and that $D_I\cap\fU_s$ is given by the equations $T_0=\dots=T_{d}=0$. We denote this element $a\in k^{\circ\circ}\setminus 0$ by $a_I$.


Let
\[S_I=\Set{\sum_{i\in I} r_i\langle D_i\rangle | r_i\in\R_{\geq 0}, \sum_{i\in I}r_i=\val(a_I)}\subset \R^{\IX} ,\]
where $(\langle D_i\rangle)_{i\in\IX}$ is regarded as the standard basis of $\R^{\IX}$.

\begin{lem}\label{lem:skeleton}
The skeleton $S_\fX$ can be identified with the union of the simplexes $S_I$ over all maximal faces $I\subset\IX$. Thus we obtain an embedding of $S_\fX$ into $\R^{\IX}$.
\end{lem}

We refer to \cite{Yu_Gromov_2014,Boucksom_Singular_2011,Kontsevich_Non-archimedean_2002,Gubler_Skeletons_2014} for related constructions.

%

\section{Calculation of vanishing cycles}\label{sec:vanishing cycles}

For any space $X$, we denote by $X\et^\sim$ the category of étale sheaves on $X$ whenever it makes sense. Let $\Lambda=\Z/\ell^\nu \Z$ for any positive integer $\nu$. We denote by $\Lambda_X$ the constant sheaf on $X$ associated to $\Lambda$.
Let $k^s$ be a separable closure of $k$, $\widehat{k^s}$ its completion, and $\widetilde{k^s}$ its residue field. For any scheme $X$ defined over $\Spec\widetilde k$, we denote $\overline X = X\times \widetilde{k^s}$. For any $k$-analytic space $X$, we denote $\overline X = X\times \widehat{k^s}$.

Let $\fX$ be a strictly semi-stable formal scheme over $k^\circ$. Let $\fXbs$ (resp.\ $\fXbe$) denote the special (resp.\ generic) fiber of the formal scheme $\overline{\fX}\coloneqq\fX\mathbin{\widehat{\otimes}_{k^\circ}}(\widehat{k^s})^\circ$ over $(\widehat{k^s})^\circ$. In \cite{Berkovich_Vanishing_1994}, Berkovich constructed two functors $\Theta  \colon {\fX_\eta}\et^\sim\rightarrow{\fX_s}\et^\sim$ and $\Psi_\eta  \colon {\fX_\eta}\et^\sim\rightarrow{\fXbs}\et^\sim$. We call them the specialization functor and the nearby cycles functor\footnote{Our terminology differs from \cite{Berkovich_Vanishing_1994}, where $\Psi$ is called the vanishing cycles functor.} respectively. We denote by $R\Theta$ and $R\Psi$ the corresponding derived functors. Our aim in this section is to compute the sheaf of nearby cycles $R\Psi\LXe$ and the sheaf of vanishing cycles $R\Phi\LXe$ defined as usual as a cone (\cite{SGA7-2} XIII 2.1).

The question being local, we only have to study the affine charts $\phi\colon \fU\rightarrow\fS(n,d,a)$ as in Definition \ref{def:strictly semi-stable formal scheme}. The formal scheme $\fS(n,d,a)$ is the completion of the scheme \[\Spec \left(k^\circ [T_0,\dots,T_{d},S^\pm_{d+1},\dots,S^\pm_n]/(T_0\cdots T_{d}-a)\right)\] along its special fiber. Therefore, by \cite{Berkovich_Vanishing_1994} Corollary 4.5(i) and Corollary 5.3, the calculation is reduced to the case of ordinary schemes.

\begin{prop}[\cite{Rapoport_Lokale_1982}]\label{prop:vanishing cycles}
We have $R^0\Psi\LXe \simeq \LXs$, an exact sequence
\[0\rightarrow\LXs\xrightarrow{\mathit{diag}} \bigoplus_{i\in\IX} \Lambda_{\overline D_i} \longrightarrow R^1 \Psi \LXe (1)\rightarrow 0,\]
and isomorphisms
\[\wedge^q R^1\Psi \LXe \simeq R^q\Psi\LXe, \quad\text{ for } q\geq 1. \]
\end{prop}

The sheaf of vanishing cycles $R\Phi\LXe$ is related to the sheaf of nearby cycles $R\Psi\LXe$ through an exact triangle
\begin{equation}\label{eq:exact triangle of vanishing cycles}
\LXs\longrightarrow R\Psi\LXe\longrightarrow R\Phi\LXe\xrightarrow{+1}.
\end{equation}

Let $\Delta^I, I\subset\IX$ be a face of $S_\fX$ and let $j \colon \overline D_I\hookrightarrow\fXbs$ denote the closed immersion. We assume that $\overline D_I$ is a projective variety. Applying $j^*$ to \eqref{eq:exact triangle of vanishing cycles}, we obtain an exact triangle
\begin{equation}\label{eq:j*exact triangle of vanishing cycles}
j^*\LXs\longrightarrow j^*R\Psi\LXe\longrightarrow j^*R\Phi\LXe\xrightarrow{+1}.
\end{equation}
Taking global sections $R\Gamma$, we obtain a long exact sequence
\begin{equation}\label{eq:long exact sequence}
\dots\rightarrow R^1\Gamma\left(j^* R\Psi\LXe\right)\xrightarrow{\beta^*} R^1\Gamma\left(j^*R\Phi\LXe\right)\xrightarrow{\alpha^*} R^2\Gamma\big(j^*\LXs\big)\rightarrow\cdots ,
\end{equation}
where we denote the two arrows above by $\beta^*$ and $\alpha^*$ respectively.

\begin{cor}\label{cor:vanishing cycles}
We have an isomorphism
\[R^1\Gamma\left(j^* R\Phi\LXe\right)\simeq \Coker\big(\Lambda\xrightarrow{\Delta}\Lambda^J\big)(-1) ,\]
where $\Delta$ denotes the diagonal map that sends an element $\lambda\in\Lambda$ to \\ $(\lambda,\dots,\lambda)\in\Lambda^J$.
Moreover, the map
\[\alpha^* \colon \Coker\big(\Lambda\xrightarrow{\Delta}\Lambda^J\big)(-1)\longrightarrow R^2\Gamma\big(j^*\LXs\big)\simeq H^2\et\big(\overline D_I,\Lambda\big)\]
is induced by the cycle class map in étale cohomology.
\end{cor}
\begin{proof}
The first statement is an immediate corollary of Proposition \ref{prop:vanishing cycles}. Let us explain the second statement. Again by comparison theorems between analytic and algebraic vanishing cycles, it suffices to work in the setting of ordinary schemes. In this proof, let us temporarily assume that $\fX$ is a scheme over $\Spec k^\circ$ instead of a formal scheme. Let
\[i_s\colon \fX_s\rightarrow\fX,\quad i_\eta\colon \fX_\eta\rightarrow\fX\]
denote the inclusions, and let
\[i_{\overline{s}}\colon \fXbs\rightarrow\fX,\quad i_{\overline{\eta}}\colon \fXbe\rightarrow\fX,\quad j_{\overline D}\colon \overline D_{I_v}\rightarrow\fX\]
denote the natural morphisms. We have an exact triangle
\begin{equation}\label{eq:triangle of pair}
i_{s!} i^!_s \LX\longrightarrow\LX\longrightarrow i_{\eta*}i^*_\eta\LX\xrightarrow{+1}.
\end{equation}
By purity, we have an isomorphism
\[R^1 i_{\eta*}i^*_\eta\LX\simeq\bigoplus_{i\in\IX}\Lambda_{D_i}(-1)\]
given by the cohomology classes of the divisors $D_i$ for $i\in\IX$.
Applying $j^*_{\overline{D}}$ to \eqref{eq:triangle of pair} and shifting by 1, we obtain an exact triangle
\[j^*_{\overline{D}}\LX \longrightarrow j^*_{\overline{D}} i_{\eta*}i^*_\eta\LX \longrightarrow j^*_{\overline{D}} i_{s!} i^!_s \LX[1] \xrightarrow{+1}.\]
We have a morphism from the exact triangle above to the exact triangle \eqref{eq:j*exact triangle of vanishing cycles}
\begin{equation*}
\begin{tikzcd}
j^*_{\overline{D}}\LX \arrow{r}{} \arrow{d}{f_1} & j^*_{\overline{D}} i_{\eta*}i^*_\eta\LX \arrow{r}{} \arrow{d}{f_2} &  j^*_{\overline{D}} i_{s!} i^!_s \LX[1] \arrow{r}{b} \arrow{d}{f_3} & j^*_{\overline{D}}\LX \arrow{d}{f_4}\\
j^*\LXs \arrow{r} & j^*R\Psi\LXe \arrow{r} & j^*R\Phi\LXe \arrow{r}{b'} & j^*\LXs[1],
\end{tikzcd}
\end{equation*}
where the maps $f_1$, $f_4$ are identities, the map $f_2$ is obtained from adjunction
\[i_{\eta *} i^*_\eta \LX \longrightarrow i_{\overline{\eta}*} i^*_{\overline{\eta}} \LX,\]
and the map $f_3$ follows from the properties of triangulated categories. Now the second statement in the corollary follows from the commutativity
\[f_4\circ b=b'\circ f_3\]
and the definition of the cycle class map in étale cohomology.
\end{proof}

\section{Deformation of analytic tubes}\label{sec:tube}

The retraction map $\tau\colon\fX_\eta\to S_\fX$ constructed by Berkovich \cite{Berkovich_Smooth_1999} is easy to describe for the standard formal scheme $\fB\coloneqq\fS(n,d,a)$ defined as in \eqref{eq:standard formal scheme}.
By Lemma \ref{lem:skeleton}, we have
\[S_{\mathfrak B}=\Set{(r_0,\dots,r_{d})\in\R^{d+1}_{\geq 0} | \sum_{i=0}^{d} r_i = \val a}\subset\R^{d+1} .\]

\begin{lem}\label{lem@retraction}
The retraction map $\tau_\fB \colon \mathfrak B_\eta\rightarrow S_\mathfrak B$ takes $x\in\mathfrak B_\eta$ to the point
\[(\val T_0(x),\dots,\val T_{d}(x))\in S_{\fB}\subset \R^{d+1} .\]
\end{lem}

For a general strictly semi-stable formal scheme $\fX$, the retraction map $\tau \colon \fX_\eta\rightarrow S_\fX$ is defined by gluing the construction above in the étale topology.
Our aim in this section is to describe the local geometry of this retraction in terms of nearby cycles.
More precisely, let $v$ be a point in $S_\fX$ sitting in the relative interior of a $d_v$-dimensional face $\Delta^{I_v}$ for $I_v=\{i_0,\dots,i_{d_v}\}\subset\IX$.
Denote by $j$ the closed immersion $\overline D_{I_v}\hookrightarrow \fXbs$. Let
\begin{equation}\label{eq:neighborhood V}
V^{s_0\dots s_{d_v}}_{I_v} = S_\fX \cap \Set{(r_0,\dots,r_N)\in\R^{\IX} | r_{i_j}\geq s_j \text{ for }j\in\{0,\dots,d_v\}} ,
\end{equation}
where $s_0,\dots,s_{d_v}\in \val k^{\circ\circ}$. It follow from the definition that 

\begin{lem}
$V^{s_0\dots s_{d_v}}_{I_v}\neq\emptyset$ if and only if $s_0+\dots+s_{d_v}\leq \val a_{I_v}$.
\end{lem}

We suppose from now on that $V^{s_0\dots s_{d_v}}_{I_v}\neq\emptyset$, and denote by $\iota$ the inclusion map $\tau^{-1}\colon (V^{s_0\dots s_{d_v}}_{I_v})\hookrightarrow \fX_\eta$.

\begin{thm}\label{thm:tube}
We have a quasi isomorphism
\[j_* j^* R\Psi \left(\LXe\right) \xrightarrow{\sim} R\Psi \left( R\iota_*\iota^* \LXe\right) .\]
\end{thm}
\begin{proof}
By adjunction, we have a morphism $\LXe\rightarrow R\iota_*\iota^*\LXe$, thus a morphism \[j_* j^* R\Psi \big(\LXe\big)\xrightarrow{\gamma} j_* j^* R\Psi \big(R\iota_*\iota^* \LXe\big).\] In order to prove the theorem, we only have to show that
\begin{enumerate}[(i)]
\item \label{item:qis}The morphism $\gamma$ is a quasi-isomorphism.
\item \label{item:support}The sheaf $R\Psi (R\iota_*\iota^* \LXe)$ is supported on $\overline D_{I_v}$.
\end{enumerate}

The properties being local, we only have to show them for the formal scheme $\fB=\fS(n,d,a)$. We assume that $d_v\leq d$ because otherwise it will not have contributions to the stratum $D_{I_v}$. The special fiber $\fB_s$ is given by the equation $T_0\cdots T_{d}=0$, and we assume that $D_{I_v}\subset \fB_s$ is further cut out by the equations $T_{d_v} = \dots = T_{d}=0$.

Now pick any elements $a_0,\dots,a_{d_v}\in k^{\circ\circ}$ such that $\val a_j = s_j$. Let $a'=a\cdot a_0^{-1}\cdots a_{d_v}^{-1}$, $\tfB=\fS(n,d,a')$. Let $f \colon \tfB\rightarrow\fB$ be the morphism given by the algebra homomorphism
\begin{multline*}
k^\circ\{T_0,\dots,T_{d},S_{d+1}^\pm,\dots,S_n^\pm\}/(T_0\cdots T_{d}-a)\longrightarrow \\
k^\circ\{T_0,\dots,T_{d},S_{d+1}^\pm,\dots,S_n^\pm\}/(T_0\cdots T_{d}-a') ,
\end{multline*}
which takes
\begin{align*}
T_j&\longmapsto a_j T_j \qquad\text{ for } j=0,\dots,d_v,\\
T_j&\longmapsto T_j \qquad\text{ for } j=d_v+1,\dots,d,\\
S_j&\longmapsto S_j \qquad\text{ for } j=d+1,\dots,n .
\end{align*}

We have
\begin{equation}\label{eq:affinoid inclusion}
Rf_{\eta *}\Lambda_{\tfB_\eta}\simeq R\iota_*\iota^* \Lambda_{\fB_\eta}.
\end{equation}
By \cite{Berkovich_Vanishing_1994} Corollary 4.5(ii), we have
\begin{equation}\label{eq:commutativity of nearby cycles}
R\Psi \Big(Rf_{\eta *} \Lambda_{\tfB_\eta}\Big)\xrightarrow{\sim}Rf_{\overline s *} \Big(R\Psi \big(\Lambda_{\tfB_\eta}\big)\Big).
\end{equation}
Combining \eqref{eq:affinoid inclusion} and \eqref{eq:commutativity of nearby cycles}, we obtain \[R\Psi \Big(R\iota_*\iota^*\Lambda_{\fB_\eta}\Big)\simeq Rf_{\overline s*}\left(R\Psi\big(\Lambda_{\tfB_\eta}\big)\right) .\] This shows (\ref{item:support}). Then (\ref{item:qis}) follows from the calculation of nearby cycles in the last section and the fact that the fibers of $f_{\overline s}$ are all cohomologically trivial.
\end{proof}

\begin{rem}
The geometry behind Theorem \ref{thm:tube} is that the étale cohomology does not change when we deform certain tubular neighborhoods.
A scheme theoretic analog can be found in \cite{Dat_Nearby_2012}.
Our approach may be adapted to give a better understanding of the result loc.\ cit.
\end{rem}

\section{Cohomological interpretation of tropical weights}\label{sec:weights}

Let $C$ be a compact quasi-smooth $k$-analytic curve, and $f \colon C\rightarrow\fX_\eta$ a $k$-analytic morphism. The image of $C$ under $\tau\circ f$ is a one-dimensional\footnote{In degenerate situations, the tropical curve $C\trop$ can be zero-dimensional. We are not interested in such cases.} polyhedral complex embedded in the skeleton $S_\fX$.
We denote it by $C\trop$, and call it the associated \emph{tropical curve}.
We call the 0-dimensional faces of $C\trop$ vertices and the 1-dimensional faces of $C\trop$ edges. We denote by $e^\circ$ the interior of an edge $e$.

\begin{prop}\label{prop:annuli}
There exists a subdivision of the edges of the tropical curve $C\trop$ such that, if we denote by $\overline {C\trop}$ the polyhedral complex after the subdivision of edges, then
\begin{enumerate}[(1)]
\item For any edge $e$ of $\overline{C\trop}$, each connected component of $(\tau\circ f)^{-1}(e^\circ)$ is isomorphic to an open annulus.
\item For each annulus $A$ as above, there exists an open subscheme $\fU$ in $\fX$, equipped with an étale morphism $\phi\colon \fU\rightarrow\fS(n,d,a)$ as in Definition \ref{def:strictly semi-stable formal scheme}, such that $f(A)$ is contained in $\fU_\eta$.
\end{enumerate} 
\end{prop}
\begin{proof}\footnote{Many thanks to Antoine Ducros for his help with this proof.}
Let us choose affine charts $\fU_1,\dots,\fU_m$ as in Definition \ref{def:strictly semi-stable formal scheme} such that $\bigcup_{i=1}^m \fU_i=\fX$. Assume that for each $i\in\{1,\dots,m\}$, the structural morphism $\fU_i\rightarrow \Spf k^\circ$ factorizes through an étale morphism
\[\phi_i\colon \fU_i\rightarrow\Spf\left(k^\circ\big\{T_{0}^{(i)},\dots,T^{(i)}_{d_i},S^{(i)\pm}_{d_i+1},\dots,S^{(i)\pm}_{n_i}\big\}/(T_{0}^{(i)}\cdots T^{(i)}_{d_i}-a_i)\right)\]
for some $0\leq d_i\leq n_i$, $a_i\in k^{\circ\circ}\setminus 0$. Let $C_i = f^{-1}(\fU_{i,\eta})$, $t_{ij} = |T^{(i)}_j\circ\phi_{i,\eta}\circ f_{|C_i}|$, for $i=1,\dots,m$, $j=0,\dots,d_i$, where $|\cdot|$ denotes the absolute value. We have the following fact concerning the variation of holomorphic functions.
\begin{lem}[\cite{Ducros_Structure_2012}(4.4.35)]\label{lem:analytic admissible graph}
For each $i=1,\dots, m$, there exists a finite graph $\Gamma_i\subset C_i$ such that the functions $t_{ij}$ are locally constant over $C_i\setminus \Gamma_i$, and piecewise linear over $\Gamma_i$.
\end{lem}
By extending $\Gamma_i$, we can assume that it contains an analytic skeleton of $C_i$ (cf.\ \cite{Ducros_Structure_2012}(5.1.8)).
Let $\Gamma$ be the union of the analytic skeleton of $C$ and $\bigcup \Gamma_i$. Let $\overline{\Gamma}$ be the convex hull of $\Gamma$, and let $\overline\Gamma_i=\overline{\Gamma}\cap C_i$. By \cite{Ducros_Structure_2012}(5.1), we have a strong deformation retraction $r\colon C\rightarrow\overline\Gamma$, and $r_{|C_i}$ gives a retraction of $C_i$ onto $\overline\Gamma_i$.
Let $K_i$ be the set of knot points\footnote{The notion of knot points is defined in \cite{Ducros_Structure_2012}(5.1.12) for an analytically admissible and locally finite subgraph inside a generically quasi-smooth $k$-analytic curve.} for the subgraph $\overline\Gamma_i\subset C_i$ for $i=1,\dots,m$, and let $K$ be the set of knot points for the subgraph $\overline\Gamma\subset C$.
Let $P_0=(\tau\circ f)\big(K\cup\bigcup_{i=1}^m K_i\big)$, and let $P_1$ be the union of the points $p\in C\trop$ such that $((\tau\circ f)_{|\overline\Gamma})^{-1}(p)$ contains an infinite number of points. The set of points $P_0$ is finite by \cite{Ducros_Structure_2012}(5.1.12.2). The set of points $P_1$ is also finite by Lemma \ref{lem:analytic admissible graph}. Therefore, the union $P\coloneqq P_1\cup P_2$ is a finite set. Now it suffices to make the subdivision of the edges of our tropical curve $C\trop$ by adding the points in $P$ as new vertices. Indeed, for each edge $e$ of the subdivided tropical curve $\overline{C\trop}$, $\overline\Gamma\cap (\tau\circ f)^{-1}(e^\circ)$ is a finite disjoint union of open segments in $\overline\Gamma$. Any such segment $s$ is by construction contained in a certain $\overline\Gamma_i$. By \cite{Ducros_Structure_2012}(5.1.12.3), $r^{-1}(s)$ is isomorphic to an open annulus, which we denote by $A$. By Lemma \ref{lem:analytic admissible graph}, we have $f_{|\overline\Gamma}\circ r_{|A} = \tau\circ f_{|A}$. So we have proved the first assertion of our proposition. Moreover, $s\subset\overline\Gamma_i$ implies that $A_i\coloneqq(r_{|C_i})^{-1}(s)$ is an open annulus inside $C_i$. We have an inclusion of two open annuli $A_i\subset A$, and both of them retract to the same segment $s$. Therefore we have $A_i=A$, and $f(A)=f(A_i)\subset\fU_{i,\eta}$. So we have proved the second assertion as well.
\end{proof}

\begin{rem}
It is pointed out by the referee that Proposition \ref{prop:annuli} is the analog of \cite[Proposition 6.4(2)]{Baker_Nonarchimedean_2011} in our global setting.
\end{rem}

From now on, we replace our tropical curve $C\trop$ by the subdivided curve $\overline{C\trop}$ produced by Proposition \ref{prop:annuli}.

We explain in the following how to equip every edge $e$ of $C\trop$ a \emph{tropical weight} $\widetilde w_e\in\Z^{\IX}$ with respect to a choice of orientation of the edge $e$.
An orientation of an edge $e$ is a choice of a direction parallel to $e$ among the two possible choices. We will see that the weight $\widetilde w_e$ is multiplied by $-1$ if we reverse the orientation of $e$.

Now fix an edge $e$ of our tropical curve $C\trop$ and an orientation of $e$.
Assume that the interior $e^\circ$ of the edge $e$ is contained in the relative interior of a $d_e$-dimensional face $\Delta^{I_e}$ for $I_e\subset\IX$.
Let $A$ be a connected component of $(\tau\circ f)^{-1}(e^\circ)$.
Fix $i\in\IX$ and denote by $p_i\colon \R^{\IX}\rightarrow\R$ the projection to the $i^\text{th}$ coordinate.
Let $r_i$ be the composition of $p_i\circ\tau\circ f_{|A}$ using the embedding $S_\fX\subset\R^{\IX}$.
By Proposition \ref{prop:annuli}(2) and the explicit description of the retraction map in Lemma \ref{lem@retraction}, the map $r_i$ equals $\val(f_i)$ for some invertible function $f_i$ on the open annulus $A$.
Choose a coordinate $z$ on $A$ such that $A$ is given by $c_1<|z|<c_2$ for two positive real numbers $c_1<c_2$ and that the image $(\tau\circ f)(z)$ moves along the orientation of the edge $e$ as $\val(z)$ increases.
Write
\[f_i=\sum_{m\in\Z} f_{i,m} z^m\]
with $f_{i,m}\in k$.
By \cite{Berkovich_Etale_1993} Lemma 6.2.2, there exists $m_i\in\Z$ with $|f_{i,m_i}| r^{m_i} > |f_{i,m}| r^m$ for all $m\neq m_i$, $c_1<r<c_2$. Therefore, for $c_1<|z|<c_2$, we have
\begin{equation}\label{eq@tropical_weight}
\val(f_i(z))=\val(f_{i,m_i})+m_i\cdot \val(z).
\end{equation}

We define the $i^\text{th}$ component $\widetilde w^i_A$ of the weight $\widetilde w_A$ to be $m_i$ for every $i\in\IX$.

\begin{lem} \label{lem:direction_of_tropical_weight}
The weight $\widetilde w_A\in\Z^{\IX}$ defined as above does not depend on the choice of the invertible function $f_i$ or the coordiante $z$. It is multiplied by $-1$ if we reverse the orientation of the edge $e$. It is parallel to the direction of the edge $e$ sitting inside $\R^{\IX}$.
In particular, it is an element of $\Ker\big(\Z^{I_e}\xrightarrow{\Sigma}\Z\big)\subset\Z^{I_\fX}$, where $\Sigma\colon\Z^{I_e}\to\Z$ sends $(x^1,\dots,x^q)\in\Z^{I_e}$ to $x^1+\dots+x^q\in\Z$.
\end{lem}
\begin{proof}
All the assertions follow from Eq.\ \eqref{eq@tropical_weight}.
\end{proof}

\begin{defin} \label{def:tropical_weight}
Let $\widetilde w_e$ be the sum of $\widetilde w_A$ over every connected component $A$ of $(\tau\circ f)^{-1}(e)$.
The element $\widetilde w_e\in\Ker\big(\Z^{I_e}\xrightarrow{\Sigma}\Z\big)\subset\Z^{I_\fX}$ is called the \emph{tropical weight} associated to the edge $e$ with the chosen orientation.
\end{defin}

\begin{rem}
Definition \ref{def:tropical_weight} is a globalization of the classical notion of tropical weights.
We refer to \cite[§6]{Baker_Nonarchimedean_2011} for similar considerations in the toric case.
\end{rem}

Let us explain the homological nature of the tropical weights.
Let $A$ be an open annulus as above and let $\phi\colon \fU\rightarrow\fS(n,d,a)$ be as in Proposition \ref{prop:annuli}(2).
We assume for simplicity that the divisors $D_0,\dots,D_d$ restricted to $\fU$ are given by the equations $T_0\circ\phi=0,\dots,T_d\circ\phi=0$ respectively, where $T_0,\dots,T_d$ are coordinates on $\fS(n,d,a)$ (see Eq.\ \eqref{eq:standard formal scheme}).
We further assume that $I_e = \{0,\dots,d_e\}$.
Let $\fB=\fS(n,d,a)$, and let $\pi_\fU\colon \fU_\eta\rightarrow\fU_s$, $\pi_\fB\colon \fB_\eta\rightarrow\fB_s$ be the reduction maps (cf.\ \cite[§1]{Berkovich_Vanishing_1994}).

\begin{lem}\label{lem:constant after reduction}
There exists a closed point $p\in\fU_s$ such that $f(A)\subset\pi_\fU^{-1}(p)$.
\end{lem}
\begin{proof}
Consider the composition $\phi_\eta\circ f_{|A} \colon A\rightarrow \fB_\eta$. In terms of affinoid algebras, it is given by $n$ power series:
\begin{align*}
T_i&\mapsto \sum_{m\in\Z} f_{i,m} z^m \qquad\text{ for }i=0,\dots,d,\\
S_i&\mapsto \sum_{m\in\Z} f_{i,m} z^m \qquad\text{ for }i=d+1,\dots,n.
\end{align*}
From the definition of the weight $\widetilde w_A$, for any $i=0,\dots,n$, $c_1<r<c_2$, $m\neq \widetilde w_A^i$, we have
\[\left| f_{i,\widetilde w_A^i} \right| r^{\widetilde w_A^i} > |f_{i,m}| r^m .\]
Therefore, if $\widetilde w_A^i=0$, then the $i^\text{th}$-coordinate of all the points in $(\pi_\fB\circ\phi_\eta\circ f)(A)$ is $\overline f_{i,0}$, where $\overline f_{i,0}$ denotes the image of $f_{i,0}$ in the residue field $\widetilde k$. If $\widetilde w_A^i\neq 0$, then the $i^\text{th}$-coordinate of all the points in $(\pi_\fB\circ\phi_\eta\circ f)(A)$ is zero. This shows that the image $(\pi_\fB\circ\phi_\eta\circ f)(A)$ is a $\widetilde k$-rational point in $\fB_s$, which we denote by $p_\fB$. By the commutativity
\[\pi_\fB\circ\phi_\eta = \phi_s\circ\pi_\fU \colon \, \fU_\eta \rightarrow \fB_s , \]
we have \[(\pi_\fU\circ f)(A)\subset \phi_s^{-1}(p_\fB). \]
Since $\phi_s$ is étale, $\phi_s^{-1}(p_\fB)$ is discrete. Then by the connectedness of the annulus $A$, there exists a point $p\in\phi_s^{-1}(p_\fB)$ such that $(\pi_\fU\circ f)(A)=p$.
\end{proof}

Since $\phi\colon \fU\rightarrow\fB$ is étale, and $\phi_s$ induces an isomorphism between the point $p$ and the point $p_\fB=\phi_s(p)$, $\phi_\eta$ induces an isomorphism between $\pi^{-1}_\fU(p)$ and $\pi_\fB^{-1}(p_\fB)$ (\cite{Berkovich_Smooth_1999} Lemma 4.4). Now $\pi^{-1}_\fB (p_\fB)$ is very easy to describe.
It is isomorphic to the generic fiber of the special formal scheme
\begin{equation}\label{eq@special_formal_scheme}
\Spf \Big(k^\circ [[T_0,\dots,T_{d_e}, S_{d_e+1},\dots,S_n]]/(T_0\cdots T_{d_e}-a')\Big).
\end{equation}

For $i\in I_e$, let $A_i$ be the generic fiber of the special formal scheme
\[\Spf\big(k^\circ[[T_i,T']]/(T_i\cdot T'-a')\big),\]
and let $c_i$ denote the morphism of $k$-analytic spaces
\[\pi_\fB^{-1}(p_\fB)\rightarrow A_i\]
induced by the homomorphism of algebras
\[k^\circ[[T_i,T']]/(T_i\cdot T'-a) \longrightarrow k[[T_0,\dots,T_{d_e},S_{d_e+1},\dots,S_n]]/(T_0\cdots T_{d_e}-a') ,\]
which takes
\begin{align*}
T_i &\longmapsto T_i\\
T' &\longmapsto T_0\cdots \widehat{T_i}\cdots T_{d_e}.
\end{align*}

\begin{lem}\label{lem:weights_as_winding_numbers}
Let $g_i=c_i\circ\phi_\eta\circ f_{|A}\colon A\rightarrow A_i$, and let $g_i^* \colon H^1\et\big(\overline{A_i},\Q_\ell\big)\rightarrow H^1\et\big(\overline A,\Q_\ell\big)$ be the induced homomorphism of étale cohomology groups.
Then $\widetilde w^i_A=g_i^*(1)\in\Q_\ell$ for all $i\in I_e$.
\end{lem}
\begin{proof}
It follows from \cite{Berkovich_Etale_1993} Lemma 6.2.5 that the winding numbers in terms of étale cohomology equals exactly the numbers $m_i$ in Eq.\ \eqref{eq@tropical_weight}.
\end{proof}

Our next step is to relate the tropical weights to vanishing cycles. Since our tropical weights are defined using étale cohomology with $\Q_\ell$-coefficients, we shall take the inverse limit over $\nu$ for all our reasonings in Sections \ref{sec:vanishing cycles} and \ref{sec:tube}. By abuse of notation, we will just replace $\Lambda$ by $\Q_\ell$, although it should be understood that the inverse limit is taken in the last step rather than from the beginning.

Let $v$, $I_v$, $V^{s_0\dots s_{d_v}}_{I_v}$, $j$ be as in Section \ref{sec:tube}. Let $A$, $\phi\colon \fU\rightarrow\fB$, $p\in\fU_s$ be as before.
We assume that $f(A)\subset\tau^{-1}(V_{I_v}^{s_0\dots s_{d_v}})$.
By Theorem \ref{thm:tube}, we have 
\[H^1\et\big(\tau^{-1}(V^{s_0\dots s_{d_v}}_{I_v})\times\widehat{k^s},\Q_\ell\big)\simeq R^1\Gamma\big(j^* R\Psi\Q_{\ell,\fX_\eta}\big) .\]

\begin{lem}We have an isomorphism
\begin{equation}\label{eq:calculation 2 of pullback to p}
R^1\Gamma\left(j_p^*R\Psi\QUe\right)\xrightarrow{\sim} H^1\et\left(\left(\tau^{-1}(V_{I_v}^{s_0\dots s_{d_v}})\cap \pi_\fU^{-1}(p)\right)\times\widehat{k^s},\Q_\ell\right) .
\end{equation}
\end{lem}
\begin{proof}
By \cite{Berkovich_Vanishing_II_1996}, we have
\begin{equation} \label{eq:vanishing cycles of pullback to p}
R^1\Gamma\left(j_p^*R\Psi\QUe\right)\xrightarrow{\sim} H^1\et\big(\pi_\fU^{-1}(p)\times\widehat{k^s},\Q_\ell\big).
\end{equation}

Using the isomorphism $\pi^{-1}_\fU(p)\simeq\pi^{-1}_\fB(p_\fB)$ and the explicit description of $\pi^{-1}_\fB(p_\fB)$ in Eq.\ \eqref{eq@special_formal_scheme}, the same arguments as in the proof of Theorem \ref{thm:tube} give an isomorphism
\begin{equation}\label{eq:deform a bit}
H^1\et\left(\left(\tau^{-1}(V_{I_v}^{s_0\dots s_{d_v}})\cap \pi_\fU^{-1}(p)\right)\times\widehat{k^s},\Q_\ell\right)\xrightarrow{\sim}H^1\et\left(\pi^{-1}_\fU (p)\times\widehat{k^s},\Q_\ell\right) .
\end{equation}
Now our lemma follows from \eqref{eq:vanishing cycles of pullback to p} and \eqref{eq:deform a bit}.
\end{proof}

Now assume that $v$ is an endpoint of the edge $e$.
Let $J\coloneqq J_{I_v}$ (see Eq.\ \eqref{eq:J}).
Since $I_e\subset J$, by Lemma \ref{lem:direction_of_tropical_weight}, the weight $\widetilde w_A$ is an element in $\Ker\big(\Z^J\xrightarrow{\Sigma}\Z\big)$.
Using the natural inclusion $\Ker(\Z^J\xrightarrow{\Sigma}\Z)\longhookrightarrow\Ker(\Q_\ell^J\xrightarrow{\Sigma}\Q_\ell\big)$, the weight $\widetilde w_A$ induces a map 
\[\widetilde w_A^* \colon \Coker\big(\Q\xrightarrow{\Delta}\Q^J\big)\longrightarrow \Q_\ell .\]

\begin{prop}\label{prop:weights-cycles}
We have the following commutative diagram
\[\begin{tikzcd}[column sep=small]
R^1\Gamma\left(j^*R\Psi\QXe\right) \arrow{r}{\beta^*} \arrow{d}{\sim} &R^1\Gamma\left(j^*R\Phi\QXe\right) \arrow{r}{\sim} & \Coker\big(\Q_\ell\xrightarrow{\Delta}\Q_\ell^J\big)(-1)\arrow{d}{\widetilde w_A^*} \\
H^1\et\big(\tau^{-1}(V_{I_v}^{s_0\dots s_{d_v}})\times\widehat{k^s},\Q_\ell\big) \arrow{r}{f_{|A}^*} & H^1\et\big(\overline{A},\Q_\ell\big) \arrow{r}{\sim} &\Q_\ell(-1) .
\end{tikzcd}\]
\end{prop}
\begin{proof}
In order to simplify notation, let us temporarily denote
\begin{align*}
&\widetilde V = \tau^{-1}(V^{s_0\dots s_{d_v}}_{I_v})\times\widehat{k^s}, \\
&\widetilde V_p = \left(\tau^{-1}(V_{I_v}^{s_0\dots s_{d_v}})\cap \pi_\fU^{-1}(p)\right)\times\widehat{k^s} ,\\
&\Q_\ell^{[I_e]} = \Coker\big(\Q_\ell\xrightarrow{\Delta}\Q_\ell^{I_e}) ,\\
&\Q_\ell^{[J]} = \Coker\big(\Q_\ell\xrightarrow{\Delta}\Q_\ell^{J}) .
\end{align*}

Let $j_p$ denote the inclusion $\overline{\{p\}}\hookrightarrow\fU_{\overline s}$. By the description of the point $p$ in the proof of Lemma \ref{lem:constant after reduction}, we have
\begin{equation}\label{eq:calculation 1 of pullback to p}
R^1\Gamma\big(j_p^*R\Psi\QUe\big) \xrightarrow{\sim} R^1\Gamma\big(j_p^*R\Phi\QUe\big)\xrightarrow{\sim} \Coker\big(\Q_\ell\xrightarrow{\Delta}\Q_\ell^{I_e})(-1) .
\end{equation}

Combining \eqref{eq:calculation 1 of pullback to p} and \eqref{eq:calculation 2 of pullback to p}, we have isomorphisms
\[H^1\et\big(\widetilde V_p,\Q_\ell\big)\xleftarrow{\sim}R^1\Gamma\big(j_p^*R\Psi\QUe\big) \xrightarrow{\sim} R^1\Gamma\big(j_p^*R\Phi\QUe\big)\xrightarrow{\sim} \Q_\ell^{[I_e]}(-1) .\]

Then the following commutative diagram follows from the cohomological interpretations of tropical weights (Lemma \ref{lem:weights_as_winding_numbers}).
\begin{equation}
\begin{tikzcd}
R^1\Gamma\big(j_p^*R\Psi\QUe\big) \arrow{r}{\sim} \arrow{d}{\sim} &R^1\Gamma\big(j_p^*R\Phi\QUe\big) \arrow{r}{\sim} & \Q_\ell^{[I_e]}(-1)\arrow{d}{\widetilde w_A^*} \\
H^1\et\big(\widetilde V_p,\Q_\ell\big) \arrow{r}{f_{|A}^*} & H^1\et\big(\overline A,\Q_\ell\big) \arrow{r}{\sim} &\Q_\ell(-1) .
\end{tikzcd}
\end{equation}

We conclude by the functoriality of the formation of vanishing cycles, i.e. the following commutative diagram:
\[\begin{tikzcd}[column sep=small]
H^1\et\big(\widetilde V,\Q_\ell\big) \arrow{d} & R^1\Gamma\left(j^*R\Psi\QXe\right) \arrow{l}[swap]{\sim}\arrow{r}{\beta^*} \arrow{d}{} &R^1\Gamma\big(j^*R\Phi\QXe\big) \arrow{r}{\sim}\arrow{d} & \Q_\ell^{[J]}(-1)\arrow{d}{} \\
H^1\et\big(\widetilde V_p,\Q_\ell\big)  & R^1\Gamma\left(j_p^*R\Psi\QUe\right) \arrow{l}[swap]{\sim} \arrow{r}{\sim}  &R^1\Gamma\big(j_p^*R\Phi\QUe\big) \arrow{r}{\sim} & \Q_\ell^{[I_e]}(-1) .
\end{tikzcd}\]
\end{proof}

\begin{rem}
Intuitively, Proposition \ref{prop:weights-cycles} says that the image of the annulus $A$ under the morphism $f$ can only go around vanishing cycles in $\tau^{-1}(V^{s_0\dots s_{d_v}}_{I_v})$ rather than arbitrary homology cycles.
\end{rem}

\section{Balancing condition in terms of étale cohomology} \label{sec:cohomological balancing}

We use the settings in Section \ref{sec:intro}.
The aim of this section is to prove the following theorem.

\begin{thm}\label{thm:cohomological balancing}
Let $\sigma_v$ denote the sum of weights around the vertex $v$ as in (\ref{eq:sum of weights}).
Let $J\coloneqq J_{I_v}$ as in Section \ref{sec:weights}.
Then $\sigma_v$ lies in the image of the following map:
\begin{align*}
\alpha_\ell \colon H^{2(\dim D_{I_v}-1)}\et\left(D_{I_v}\times_{\widetilde k} \Spec \widetilde{k^s},\Q_\ell\right)(\dim D_{I_v}-1)&\longrightarrow \Ker\big(\Q_\ell^J\xrightarrow{\Sigma}\Q_\ell\big) ,\\
L &\longmapsto \big(\, L\cdot L_j,\ j\in J \,\big),
\end{align*}
where $\alpha_\ell$ is dual to the map $\alpha^*$ in Corollary \ref{cor:vanishing cycles} for $\mathbb Q_\ell$ coefficients.
\end{thm}

We begin with a simple observation.

\begin{lem}
For any extension $k\subset k'$ of non-archimedean fields, if we apply extension of ground fields to $\fX_\eta$ and to the map $f \colon C\rightarrow \fX_\eta$, the skeleton $S_\fX$ and the tropical curve $C\trop$ remains unchanged.
\end{lem}

We assume that our tropical curve $C\trop$ is already subdivided according to Proposition \ref{prop:annuli}.
We choose a sufficiently small convex neighborhood $V$ of $v$ inside $S_\fX$ which does not contain any other vertex of $C\trop$. We further require that $V$ is of the form \eqref{eq:neighborhood V}. This can be achieved by making a finite extension of our field $k$.

\begin{lem}\label{lem:boundary}
Let $k'$ be a separably closed non-archimedean field, $C^\circ$ a connected compact quasi-smooth $k'$-analytic curve, and let $b_1,\dots,b_m$ be the boundary points of $C^\circ$. 
We assume that there are neighborhoods $A_1,\dots,A_m$ of $b_1,\dots,b_m$ which are pairwise disjoint and are isomorphic to annuli. Let $A=\coprod A_i$ and let $\iota_b$ denote the inclusion $A\hookrightarrow C^\circ$. Then the composition $\Sigma\circ \iota_b^*$
\[H^1\et(C^\circ,\Q_\ell)(1)\xrightarrow{\iota_b^*} H^1\et(A,\Q_\ell)(1)\simeq \Q_\ell^m\xrightarrow{\Sigma}\Q_\ell\]
is zero, where $\Sigma$ denotes the sum.
\end{lem}
\begin{proof}
By gluing discs onto $b_1,\dots,b_m$, we embed $C^\circ$ into a proper smooth $k'$-analytic curve, which we denote by $\widehat{C^\circ}$.
Let $z_1,\dots,z_m$ be the centers of the discs, and let $Z=\coprod_{i=1}^m z_i$.
Now the lemma follows from the Gysin exact sequence
\begin{multline*}
0\rightarrow H^1\et(\widehat{C^\circ}, \Q_\ell)(1)\rightarrow H^1\et(C^\circ,\Q_\ell)(1)\rightarrow H^0\et (Z,\Q_\ell)\\ \rightarrow H^2\et(\widehat{C^\circ},\Q_\ell)(1)\simeq \Q_\ell \rightarrow \cdots .
\end{multline*}
\end{proof}

Let $r$ be a positive real number.
Let \[V^r=\Set{x\in V\subset S_\fX\subset \R^{\IX} | \dist (x, V^c)>r},\]
where $\dist$ denotes the standard Euclidean metric in $\R^{\IX}$ and $V^c$ denotes the complement of $V$ in $S_\fX$. Put $V^b=V\setminus V^r$, $C^\circ = (\tau\circ f)^{-1}(V)$, $C^b=(\tau\circ f)^{-1}(V^b)$, $\iota_b\colon C^b\hookrightarrow C^\circ$. We choose $r$ sufficiently small, such that $C^b$ does not contain any vertices of $C\trop$. For each segment $e$ in $C\trop\cap V^b$, we choose the orientation of $e$ to be the one that points away from $v$.
The definition of tropical weights in Section \ref{sec:weights} gives an element $\widetilde w_e\in \Ker\big(Z^{J}\xrightarrow{\Sigma}Z\big)$, which induces a map \[\widetilde w_e^* \colon \Coker\big(\Q_\ell\xrightarrow{\Delta}\Q_\ell^J\big)\longrightarrow \Q_\ell.\]
Then $\sigma_v$ is the sum of $\widetilde w_e$ over all segments $e$ in $C\trop\cap V^b$.
It induces a map \[\sigma_v^* \colon \Coker\big(\Q_\ell\xrightarrow{\Delta}\Q_\ell^J\big)\longrightarrow \Q_\ell.\] By Proposition \ref{prop:weights-cycles}, we have a commutative diagram
\begin{equation}\label{eq:commutative diagram for C}
\begin{tikzcd}[column sep=small]
R^1\Gamma\left(j^*R\Psi\QXe\right) \arrow{r}{\beta^*} \arrow{d}{\sim} &R^1\Gamma\left(j^*R\Phi\QXe\right) \arrow{r}{\sim} & \Coker\big(\Q_\ell\xrightarrow{\Delta}\Q_\ell^J\big)(-1)\arrow{d}{\sigma_v^*} \\
H^1\et(\tau^{-1}(V)\times\widehat{k^s},\Q_\ell) \arrow{r}{f_{|C^b}^*} & H^1\et(\overline{C^b},\Q_\ell) \arrow{r}{\Sigma} &\Q_\ell(-1),
\end{tikzcd}
\end{equation}
where the bottom row factorizes as
\[H^1\et(\tau^{-1}(V)\times\widehat{k^s},\Q_\ell) \xrightarrow{f_{|C^\circ}^*} H^1\et(\overline{C^\circ},\Q_\ell)\xrightarrow{\iota_b^*} H^1\et(\overline{C^b},\Q_\ell) \xrightarrow{\Sigma} \Q_\ell(-1).\]
By Lemma \ref{lem:boundary}, we have $\Sigma\circ \iota_b^*=0$. By the commutativity of \eqref{eq:commutative diagram for C}, we have 
\begin{equation}\label{eq:sigmabeta=0}
\sigma_v^*\circ\beta^*=0 .
\end{equation}
By \eqref{eq:long exact sequence}, we have a long exact sequence
\begin{equation}\label{eq:long exact sequence Q_l coefficients}
\dots\rightarrow R^1\Gamma(j^* R\Psi\QXe)\xrightarrow{\beta^*} R^1\Gamma\left(j^*R\Phi\QXe\right)\xrightarrow{\alpha^*} R^2\Gamma(j^*\Q_{\ell,\fXbs})\rightarrow\cdots .
\end{equation}
Combining \eqref{eq:sigmabeta=0} and the exactness of \eqref{eq:long exact sequence Q_l coefficients}, we have proved that the sum of weights $\sigma_v$ lies in the image of the map $\alpha_\ell$.

\section{From cohomological classes to algebraic cycles}\label{sec:From cohomological classes to algebraic cycles}

The passage from Theorem \ref{thm:cohomological balancing} to Theorem \ref{thm:balancing condition} is a simple application of the standard conjectures on algebraic cycles, which is easy to prove in codimension one.
More precisely, for divisors with rational coefficients in a projective variety, Matsusaka \cite{Matsusaka_Criteria_1957} proved that numerical equivalence implies algebraic equivalence. So in particular, numerical equivalence implies homological equivalence.

Let us suppose that $\Image (\alpha_\Q)\subset \Q^{I_\fX}$ is contained in a hyperplane given by $f=0$.
Let $x_i,\  i\in I_\fX$ be the coordinates on $\Q^{I_\fX}$ and write $f$ as $\sum_{i\in {I_\fX}} a_i x_i$, for some $a_i\in\Q$.
The fact that $\Image(\alpha_\Q)$ is contained in the hyperplane $f=0$ implies that the $\Q$-divisor $\sum_{i\in {I_\fX}} a_i L_i$ is numerically equivalent to 0, it is thus homologically equivalent to 0.
Therefore, the image $\Image(\alpha_\ell)$ in Theorem \ref{thm:cohomological balancing} is also contained in the hyperplane in $\Q_\ell^{I_\fX}$ defined by the same equation $f=0$. To conclude, we have shown that $\Image \alpha_\Q\otimes \Q_\ell\simeq \Image \alpha_\ell$. So we have deduced Theorem \ref{thm:balancing condition} from Theorem \ref{thm:cohomological balancing}.


\bibliographystyle{plain}
\bibliography{../dahema}

\end{document}